\documentclass[a4paper,12pt]{article}
\usepackage{hyperref}
\hypersetup{
    colorlinks=true,
    linktoc=all,
    linkcolor=red,     
    citecolor=blue,     
}
\usepackage{graphicx}
\usepackage{amsmath,amssymb,amsfonts,amsthm,graphics,enumitem}
\usepackage{latexsym, bm}
\usepackage{multicol}
\usepackage{indentfirst}

\usepackage{setspace}
\newtheorem{theorem}{Theorem}[section]
\newtheorem{lemma}{Lemma}[section]

\newtheorem{corollary}{Corollary}[section]

\newtheorem{observation}{Observation}[section]
\newcommand{\ignore}[1]{}
\newtheorem{claimnum}{Claim}

\begin{document}
\noindent
\begin{spacing}{1}

\title {Edge and spectral conditions for rainbow bipancyclicity in bipartite graph collections}
\date{}

\author{
Lihua You\footnote{School of Mathematical Sciences, South China Normal University, Guangzhou, 510631, P. R. China.
E-mail: {\tt ylhua@scnu.edu.cn}.} \;\;   \;\; Xiaoxue Zhang\footnote{Corresponding author. School of Mathematical Sciences, South China Normal University, Guangzhou, 510631, P. R. China.
E-mail: {\tt zhang\_xx1209@163.com}.}\;\;   \;\; Xinghui Zhao\footnote{School of Mathematical Sciences, South China Normal University, Guangzhou, 510631, P. R. China.
E-mail: {\tt xhzhao@m.scnu.edu.cn}.}
}
\maketitle
\begin{abstract}
Let $\mathbf{G}=\{G_1,\dots,G_{2n}\}$ be a collection of not necessarily distinct balanced bipartite graphs of order $2n$ with a common bipartition $V=(X,Y)$. A cycle $C$ with $V(C)\subseteq V$ and $|E(C)|\leq 2n$ is called \emph{rainbow} in $\mathbf{G}$, if there exists an injection $\phi\colon E(C)\to [2n]$ such that $e\in E(G_{\phi(e)})$ for each $e\in E(C)$.
The graph collection $\mathbf{G}$ is said to be \emph{rainbow bipancyclic} if it contains a rainbow cycle of every even length from 4 to $2n$.
In this paper, we show that if $e(G_i)\ge n^2-n+1$ for each $i\in[2n]$, then $\mathbf{G}$ is rainbow bipancyclic unless $G_1=G_2=\dots=G_{2n}\cong B_{n}^1$, where $B_{n}^1$ is the graph obtained from $K_{n,n}$ by deleting all edges in its one subgraph $K_{n-1,1}$. As a corollary, we obtain that if $\rho(G_i)\ge \rho(B_{n}^1)>\sqrt{n(n-1)}$ for each $i\in[2n]$, then $\mathbf{G}$ is rainbow bipancyclic unless $G_1=G_2=\dots=G_{2n}\cong B_{n}^1$, which strengthens Theorem $5$ of [Electron. J. Combin., \textbf{33}(3)(2026), \#P3.29] and Theorem $1.6$ of [Linear Algebra Appl., \textbf{743}(2026), 250--270] from rainbow Hamiltonicity to rainbow bipancyclicity.
We also characterize all bipartite graph collections that are not rainbow bipancyclic under the condition $\rho(G_i)\ge \sqrt{n(n-1)}$ for each $i\in[2n]$.


\end{abstract}

\section{Introduction}

\subsection{Hamiltonicity and bipancyclicity}
A graph is Hamiltonian if it contains a cycle covering all its vertices exactly once. The problem of determining whether a graph is Hamiltonian has long attracted considerable research attention. As a special and important family of graphs, bipartite graphs are widely investigated. There have been many results concerning the Hamiltonicity of bipartite graphs. In 1963, Moon and Moser \cite{Moon} gave an edge condition under a minimum degree hypothesis for a balanced bipartite graph to be Hamiltonian. In 2016, Li and Ning \cite{Ning} obtained the spectral analogue of Moon-Moser's theorem. Herein, we present a special version of these results. Let $B_{n}^1$ be the graph obtained from $K_{n,n}$ by deleting all edges in its one subgraph $K_{n-1,1}$. In particular, $B_{1}^1\cong K_{1,1}$ and $B_{2}^1\cong P_4$.

\begin{theorem}[\upshape\cite{Moon}]\label{thm1-1}
Let $n\ge 2$ and $G$ be a balanced bipartite graph of order $2n$. If $e(G)\ge n^2-n+2$, then $G$ is Hamiltonian.
\end{theorem}

\begin{theorem}[\upshape\cite{Ning}]\label{thm1-2}
Let $n\ge 4$ and $G$ be a balanced bipartite graph of order $2n$. If $\rho(G)\ge \rho(B_{n}^1)$, then $G$ is Hamiltonian unless $G\cong B_{n}^1$.
\end{theorem}

A natural and much stronger notion than Hamiltonicity is that of bipancyclicity.
A balanced bipartite graph of order $2n$ is \emph{bipancyclic} if it contains a cycle of every even length from 4 to $2n$. Mitchem and Schmeichel \cite{Mitchem} proved that under the same conditions as Theorem \ref{thm1-1}, $G$ is bipancyclic.

\begin{theorem}[\upshape\cite{Mitchem}]\label{thm1-3}
Let $n\ge 2$ and $G$ be a balanced bipartite graph of order $2n$. If $e(G)\ge n^2-n+2$, then $G$ is bipancyclic.
\end{theorem}

\begin{theorem}[\upshape\cite{Adamus}]\label{thm1-4}
Let $G$ be a balanced bipartite graph of order $2n$. If $e(G)\ge n^2-n+1$, then $G$ is bipancyclic unless $G\cong B_{n}^1$.
\end{theorem}

Motivated by Theorems \ref{thm1-1} and \ref{thm1-3}, it is natural to conjecture that the condition in Theorem \ref{thm1-2} also guarantees the bipancyclicity of $G$. To this end, we first present a preliminary result that establishes the relation between the number of edges and the spectral radius of bipartite graphs.

\begin{theorem}[\upshape\cite{Bhattacharya}]\label{thm1-5}
For every bipartite graph $G$ with $e(G)\ge 1$, we have $\rho^2(G)\le e(G)$, with equality if and only if $G$ is the disjoint union of a complete bipartite graph and isolated vertices.
\end{theorem}

By Theorems \ref{thm1-4} and \ref{thm1-5}, we immediately obtain the following corollary.

\begin{corollary}\label{co1-1}
Let $G$ be a balanced bipartite graph of order $2n$. If $\rho(G)\ge \rho(B_{n}^1)$, then $G$ is bipancyclic unless $G\cong B_{n}^1$.
\end{corollary}

\begin{proof}
Since $B_{n}^1$ is connected and $K_{n-1,n}\cup K_1$ is a proper subgraph of $B_{n}^1$, we have $\rho(B_{n}^1)>\rho(K_{n-1,n}\cup K_1)=\sqrt{n(n-1)}$. Since $\rho(G)\ge \rho(B_{n}^1)$, by Theorem \ref{thm1-5}, it follows that $e(G)>n^2-n$, which implies that $G$ is bipancyclic unless $G\cong B_{n}^1$ by Theorem \ref{thm1-4}.
\end{proof}

\subsection{Main results}
In recent years, the existence of rainbow subgraphs in graph collections has become a popular research topic. Let $\mathbf{G}=\{G_1,\dots,G_s\}$ be a collection of not necessarily distinct graphs with a common vertex set $V$. For a graph $H$ with $V(H)\subseteq V$ and $|E(H)|\leq s$, if there exists an injection $\phi\colon E(H)\to [s]$ such that $e\in E(G_{\phi(e)})$ for each $e\in E(H)$, in other words, each edge of $H$ comes from a distinct graph $G_i$ with $i\in[s]$, then we say that $H$ is \emph{rainbow} in $\mathbf{G}$. In particular, a collection of graphs with a common vertex set is said to be \emph{rainbow Hamiltonian} if it contains a rainbow Hamiltonian cycle.

Many classical theorems have been investigated in the context of graph collections, yielding their rainbow analogues.
More work on graph collections is already available, including general graph collections \cite{Cheng2026,Cheng,Guo,Guo2026,He,Joos,Li2023,Li2024,LiandWang,LiuandChen,MaandZhang,Sun,You,Zhang2026,Zhang}, bipartite graph collections \cite{Bradshaw,Chen,Chen2026,He2026,Hu,Ma}, directed graph collections \cite{Babinski,Chakraborti,Gerbner,Li2025}, hypergraph collections \cite{Cheng2023,Cheng2025,Gao2022,Lu2022,Lu2021} and so on. For more results, we refer the readers to the survey \cite{SunandWang}.

A collection of graphs of order $n$ with a common vertex set is said to be \emph{rainbow pancyclic} if it contains a rainbow cycle of every length from $3$ to $n$.
In 2020, Joos and Kim \cite{Joos} proved that if $n\ge 3$ and $\mathbf{G}=\{G_1,\dots, G_n\}$ is a collection of graphs of order $n$ with a common vertex set $V$ such that $\delta(G_i)\geq\frac{n}{2}$ for each $i\in[n]$, then there exists a rainbow Hamiltonian cycle in $\mathbf{G}$. This result can be regarded as a generalization of Dirac's theorem, since it reduces to Dirac's theorem when $G_1=\dots=G_n$.
In 2024, Li, Li and Li \cite{Li2024} showed that the same degree condition further guarantees rainbow pancyclicity of $\mathbf{G}$ unless $G_1=\dots=G_n\cong K_{\frac{n}{2},\frac{n}{2}}$. In 2025, Zhang and van Dam \cite{Zhang}
presented sufficient conditions in terms of edge number and spectral
radius, respectively, for graph collections to admit rainbow Hamilton cycles. You, Zhang and Zhao \cite{You} strengthened their results from rainbow Hamiltonicity to rainbow pancyclicity.
Recently, Liu, Chen and Ma \cite{LiuandChen} established a stronger Ore-type condition for the existence of rainbow Hamiltonian cycles in graph collections. Under the same Ore-type condition, Li, Wang and Yan \cite{LiandWang} further considered the rainbow pancyclicity.

A bipartite graph with the bipartition $V=(X,Y)$ is called \emph{balanced} if $|X| =|Y|$.
A collection of balanced bipartite graphs of order $2n$ with a common bipartition $V=(X,Y)$ is said to be \emph{rainbow bipancyclic} if it contains a rainbow cycle of every even length from 4 to $2n$.
In 1963, Moon and Moser \cite{Moon} established a Dirac-type condition for Hamiltonicity in balanced bipartite graphs. In 1982, Schmeichel and Mitchem \cite{Schmeichel} further extended this result to bipancyclicity. In 2021, Bradshaw \cite{Bradshaw} generalized both classical results above to balanced bipartite graph collections and obtained their corresponding rainbow analogues.
In 2026, Chen, Liu and Yuan \cite{Chen2026} generalized Theorem \ref{thm1-2} to bipartite graph collections and obtained a rainbow version of Theorem \ref{thm1-2} for $n\ge2$. Independently, He and Lu \cite{He2026} obtained the same conclusion for $n\ge 4$.

\begin{theorem}[\upshape\cite{Chen2026,He2026}]\label{thm1-6}
Let $n\ge 2$ and $\mathbf{G}=\{G_1, \dots, G_{2n}\}$ be a collection of not necessarily distinct balanced bipartite graphs of order $2n$ with a common bipartition $V=(X,Y)$. If $\rho(G_i)\ge \rho(B_{n}^1)$ for each $i\in [2n]$, then $\mathbf{G}$ contains a rainbow Hamiltonian cycle unless $G_1=G_2=\dots=G_{2n}\cong B_{n}^1$.
\end{theorem}

Motivated by the aforementioned works, it is natural to ask whether Theorem \ref{thm1-6} can be extended to rainbow bipancyclicity. In this paper, we give an affirmative answer to this question, which is stated as Corollary \ref{cor1-2}. In fact, we obtain a stronger result, namely Theorem \ref{thm1-7}.

\begin{theorem}\label{thm1-7}
Let $\mathbf{G}=\{G_1, \dots, G_{2n}\}$ be a collection of not necessarily distinct balanced bipartite graphs of order $2n$ with a common bipartition $V=(X,Y)$. If $e(G_i)\ge n^2-n+1$ for each $i\in [2n]$, then $\mathbf{G}$ is rainbow bipancyclic unless $G_1=G_2=\dots=G_{2n}\cong B_{n}^1$.
\end{theorem}

By Theorems \ref{thm1-5}, \ref{thm1-7} and the fact that $\rho(B_{n}^1)>\sqrt{n(n-1)}$, the following corollary is immediate.

\begin{corollary}\label{cor1-2}
Let $\mathbf{G}=\{G_1, \dots, G_{2n}\}$ be a collection of not necessarily distinct balanced bipartite graphs of order $2n$ with a common bipartition $V=(X,Y)$. If $\rho(G_i)\ge \rho(B_{n}^1)$ for each $i\in [2n]$, then $\mathbf{G}$ is rainbow bipancyclic unless $G_1=G_2=\dots=G_{2n}\cong B_{n}^1$.
\end{corollary}

It can be seen that Theorem \ref{thm1-7} and Corollary \ref{cor1-2} generalize Theorem \ref{thm1-4} and Corollary \ref{co1-1} from bipartite graphs to bipartite graph collections, respectively.

Noting that $\rho(B_{n}^1)>\sqrt{n(n-1)}$, we relax the condition in Corollary \ref{cor1-2} to $\rho(G_i)\ge \sqrt{n(n-1)}$ for each $i\in[2n]$, and obtain Theorem \ref{thm1-8}.
For the convenience of stating the result, we introduce the following notation. For each $v\in V=(X,Y)$ with $|X|=|Y|=n$, let $B_v$ ($\cong K_{n-1,n}\cup K_1$) denote the balanced bipartite graph on vertex set $V$ such that $v$ is an isolated vertex and $V\setminus\{v\}$ induces a bipartite graph $K_{n-1,n}$. For distinct $u,v\in V$ in different partite sets, let $B_{v,u}$ ($\cong B_{n}^1$) be the graph obtained from $B_v$ by adding the edge $vu$.

\begin{theorem}\label{thm1-8}
Let $n\ge 2$ and $\mathbf{G}=\{G_1, \dots, G_{2n}\}$ be a collection of not necessarily distinct balanced bipartite graphs of order $2n$ with a common bipartition $V=(X,Y)$. If $\rho(G_i)\ge \sqrt{n(n-1)}$ for each $i\in [2n]$, then $\mathbf{G}$ is not rainbow bipancyclic if and only if one of the following holds:
\begin{enumerate}[label=(\roman*), font=\upshape, itemsep=2pt, align=left, leftmargin=1em]
 \item there exists a vertex $v\in V$ such that $|\{i\in[2n]:G_i=B_v\}| \ge 2n-1$;
 \item there exist distinct $u,v\in V$ in different partite sets such that for each $i\in[2n]$,
 \[G_i\in
 \begin{cases}
 \{B_v,B_w,B_{v,u}\}, \text{where } w\,(\neq u) \text{ is in the same partite set as } u, & \text{if } n=2;\\[4pt]
 \{B_v,B_{v,u}\}, & \text{if } n\ge 3.
 \end{cases}
 \]
\end{enumerate}
\end{theorem}

\subsection{Notation and organization}
Here, we introduce some basic notation and concepts in graph theory. All graphs considered in this paper are simple, finite and undirected.
Let $G=(V(G),E(G))$ be a simple graph with vertex set $V(G)$ and edge set $E(G)$, where $|E(G)|=e(G)$. The \emph{adjacency matrix} $A(G)=(a_{uv})_{u,v\in V(G)}$ of $G$ is a $|V(G)|\times |V(G)|$ matrix, where $a_{uv}=1$ if $uv\in E(G)$, and $a_{uv}=0$ otherwise. The maximum of the moduli of the eigenvalues of $A(G)$ is called the \emph{spectral radius} of $G$ and denoted by $\rho(G)$. A \emph{matching} in $G$ is a set of pairwise disjoint edges.

For a vertex $u\in V(G)$, the set of neighbours of $u$ in $G$ is denoted by $N_{G}(u)$. For $U\subseteq V(G)$, the \emph{induced subgraph} $G[U]$ of $G$ satisfies $V(G[U])=U$ and $E(G[U])=\{uv\in E(G): u,v\in U\}$, and we say that the vertex subset $U$ induces the subgraph $G[U]$.
A graph $H$ is a \emph{subgraph} of a graph $G$, denoted by $H \subseteq G$,
if $V(H)\subseteq V(G)$ and $E(H)\subseteq E(G)$. We call $H$ a \emph{proper subgraph} of $G$, denoted by $H \subsetneq G$, if $H\subseteq G$ and $H \neq G$.
The \emph{union} of two graphs $H$ and $G$ is the graph $H\cup G$ with vertex set $V(H)\cup V(G)$ and edge set $E(H)\cup E(G)$.

Let $G$ be a bipartite graph with bipartition $(X,Y)$. We use $\widehat{G}$ to denote the \emph{quasi-complement} of $G$, that is, the graph with vertex set $V(\widehat{G})=V(G)$ such that for any $x\in X$ and $y\in Y$, $xy\in E(\widehat{G})$ if and only if $xy\notin E(G)$.
We use $P_n$, $K_n$ and $K_{a,b}$ to denote the path of order $n$, the complete graph of order $n$ and the complete bipartite graph with partitions of size $a$ and $b$, respectively. A cycle of order $n$ is called an $n$-cycle.
For a positive integer $n$, let $[n]=\{1,2,\dots,n\}$, and we set $[0]=\emptyset$. For any terminology and notation that are not explicitly defined, we refer the reader to \cite{Bondy}.


\section{Proof of Theorem \ref{thm1-7} }
This section will be dedicated to proving Theorem \ref{thm1-7}. We first present several auxiliary results that will be used later.

\begin{theorem}[\upshape\cite{Hall}]\label{thm2-1}
Let $B$ be a bipartite graph with bipartition $(X,Y)$. There is a matching which covers every vertex in $X$ if and only if $|N_B(S)|\ge |S|$ for all $S\subseteq X$.
\end{theorem}

The following basic observation follows easily from Theorem \ref{thm2-1}, and we omit its proof.

\begin{observation}\label{ob2-1}
Let $0< a\le b$. If a bipartite graph is obtained from $K_{a,b}$ by deleting at most one edge, then it has a matching covering all vertices in the partition of size $a$, except when $a=b=1$ and the only edge is removed.
\end{observation}

Now we introduce some notation that will be used throughout the subsequent proofs.

Let $\mathbf{G}=\{G_1, \dots, G_{2n}\}$ be a collection of not necessarily distinct balanced bipartite graphs of order $2n$ with a common bipartition $V=(X,Y)$.
Let $G=\bigcup\limits_{i=1}^{2n} G_i$ be a simple bipartite graph with $V(G)=V$ and $E(G)=\bigcup\limits_{i=1}^{2n}E(G_i)$, and let $\widehat{G}$ be the quasi-complement of $G$. Then $E(\widehat{G})=E(K_{n,n})\setminus E(G)=\bigcap\limits_{i=1}^{2n} E(\widehat{G_i})$, and we write $e(\widehat{G})=m$.

For $2\le\ell\le n$, let $\mathcal{C}_{2\ell}(G)$ be the set of all $2\ell$-cycles in $G$ with $|\mathcal{C}_{2\ell}(G)|=N_{2\ell}$.
For an edge $e\in E(G)$, let $N_{e,2\ell}=|\{C\in\mathcal{C}_{2\ell}(G):e\in E(C)\}|$.

\begin{lemma}\label{lem2-1}
Suppose $N_{2\ell}>0$ and $m\le n-2$. Then $N_{e,2\ell}\le \frac{2N_{2\ell}}{n-m}$ for every $e\in E(G)$.
\end{lemma}
\begin{proof}
Fix $e=x_0y_0\in E(G)$, where $x_0\in X$ and $y_0\in Y$. The number of $2\ell$-cycles not containing $e$ in $G$ is $N_{2\ell}-N_{e,2\ell}$. In the following, we give a lower bound for $N_{2\ell}-N_{e,2\ell}$ in terms of $N_{e,2\ell}$.

Firstly, for any $C\in \mathcal{C}_{2\ell}(G)$ containing $e$ with
$C=x_0y_0x_1y_1\cdots x_{\ell-1}y_{\ell-1}x_0$, we define two types of cycle-switching operations as follows.

\vspace{6pt}
\noindent\textbf{Operation I:} The vertex set $V(C)$ remains unchanged. For each $j\in[\ell-2]$, delete $x_0y_0$, $y_jx_{j+1}$, and add $x_0y_j$, $y_0x_{j+1}$.
\vspace{6pt}

\vspace{6pt}
\noindent\textbf{Operation II:} Replace $y_0$ with a vertex in $Y\setminus V(C)$. For each $z\in Y\setminus V(C)$, delete $x_0y_0$, $y_0x_1$, and add $x_0z$, $zx_1$.
\vspace{6pt}

Secondly, we prove that for any $2\ell$-cycle $C\in\mathcal{C}_{2\ell}(G)$ with $e\in E(C)$, there exist at least $n-2-m$ cycle-switching operations such that each operation yields a $2\ell$-cycle in $G$ that does not contain the edge $e$.

For Operation I, if the two added new edges $x_0y_j,y_0x_{j+1}\in E(G)$, then the result is the $2\ell$-cycle $x_0y_jx_{j}y_{j-1}\cdots x_1y_0x_{j+1}y_{j+1}\cdots x_{\ell-1}y_{\ell-1}x_0\in \mathcal{C}_{2\ell}(G)$, which does not contain $e$. There are $\ell-2$ possible cycle-switching operations since $j\in[\ell-2]$.

For Operation II, if the two added new edges $x_0z,zx_1\in E(G)$, the result is the $2\ell$-cycle $x_0zx_1y_1\cdots x_{\ell-1}y_{\ell-1}x_0\in \mathcal{C}_{2\ell}(G)$ avoiding $e$. There are $n-\ell$ possible cycle-switching operations since $z\in Y\setminus V(C)$.

Thus there exist $(\ell-2)+(n-\ell)=n-2$ possible cycle-switching operations. Note that all edges added by these operations are pairwise distinct, which implies each edge in $E(\widehat{G})$ can invalidate at most one cycle-switching operation. It follows that $C$ admits at least $n-2-m$ cycle-switching operations satisfying the requirement.

Thirdly, we prove that any $2\ell$-cycle $C'\in\mathcal{C}_{2\ell}(G)$ with $e\notin E(C')$ can be obtained from at most two $2\ell$-cycles containing $e$ by cycle-switching operations I and II.

If $y_0\in V(C')$, then either $C'$ is obtained by Operation I, or it cannot be obtained by any of the cycle-switching operations I or II. Assume $C'$ arises from Operation I. There are two paths between $x_0$ and $y_0$ in $C'$, denoted by $P=x_0v_1u_1\cdots v_ru_ry_0$ and $Q=x_0v'_1u'_1\cdots v'_su'_sy_0$. Then $C'$ can be obtained via Operation I from $C''=C'-\{x_0v_1,y_0u'_s\}+\{x_0y_0,v_1u'_s\}$ or $C'''=C'-\{x_0v'_1,y_0u_r\}+\{x_0y_0,v'_1u_r\}$.

If $y_0\notin V(C')$, then either $C'$ is obtained by Operation II, or it cannot be obtained by any of the cycle-switching operations I or II. Assume $C'$ arises from Operation II, and $N_{C'}(x_0)=\{v_1,v_2\}$. Then $C'$ can be obtained via Operation II from the cycle constructed by replacing $v_1$ or $v_2$ with $y_0$.

Finally, combining the above arguments, $G$ contains at least $\frac{(n-2-m)N_{e,2\ell}}{2}$ $2\ell$-cycles not containing $e$. It follows that $N_{2\ell}-N_{e,2\ell}\ge \frac{(n-2-m)N_{e,2\ell}}{2}$, and thus $N_{e,2\ell}\le \frac{2N_{2\ell}}{n-m}$.

This completes the proof.
\end{proof}

\begin{lemma}\label{lem2-2}
Let $e(G_i)\ge n^2-n+1$ for each $i\in [2n]$. Then
\begin{enumerate}[label=(\roman*), font=\upshape, itemsep=2pt, align=left, leftmargin=1em]
 \item $m\le n-1$;
 \item for each $2\le\ell\le n$, $G$ contains a $2\ell$-cycle if and only if $\mathbf{G}$ contains a rainbow $2\ell$-cycle;
 \item $G$ is bipancyclic if and only if $\mathbf{G}$ is rainbow bipancyclic.
\end{enumerate}
\end{lemma}

\begin{proof}
Since $e(G_i)\ge n^2-n+1$ for each $i\in [2n]$, we have $e(\widehat{G_i})=e(K_{n,n})-e(G_i)\le n-1$. Recalling that $E(\widehat{G})=E(K_{n,n})\setminus E(G)=\bigcap\limits_{i=1}^{2n} E(\widehat{G_i})$ with $e(\widehat{G})=m$, we have
$E(\widehat{G})\subseteq E(\widehat{G_i})$, and thus $m\le n-1$, (i) holds.

Now we show that (ii) holds. The sufficiency is trivial, so it suffices to prove the necessity. Fix $\ell\in\{2,\dots,n\}$ and suppose $G$ contains a $2\ell$-cycle.

\vspace{6pt}
$\mathbf{Case~1.}$ $m=n-1$.
\vspace{6pt}

Then $E(\widehat{G_i})=E(\widehat{G})$ for each $i\in [2n]$, and thus $G_1=\cdots=G_{2n}=G$. Therefore, every cycle in $G$ is rainbow in $\mathbf{G}$, and thus (ii) holds.

\vspace{6pt}
$\mathbf{Case~2.}$ $m\le n-2$.
\vspace{6pt}

Firstly, we define $\omega(e)=|\{i\in[2n]:e\notin E(G_i)\}|$ for each $e\in E(G)$ and $\omega(C)=\sum\limits_{e\in E(C)}\omega(e)$ for any $C\in\mathcal{C}_{2\ell}(G)$.
Then
\begin{equation}\label{e1}
\begin{aligned}
\sum_{e\in E(G)}\omega(e)
&=\sum_{i=1}^{2n}|E(G)\setminus E(G_i)|\\
&=\sum_{i=1}^{2n}|E(K_{n,n})\setminus(E(\widehat{G})\cup E(G_i))|\\
&=\sum_{i=1}^{2n}|E(\widehat{G_i})\setminus E(\widehat{G})|\\
&\le 2n(n-1-m),
\end{aligned}
\end{equation}
and
\begin{equation*}
\sum_{C\in\mathcal{C}_{2\ell}(G)}\omega(C)=\sum_{C\in\mathcal{C}_{2\ell}(G)}\sum_{e\in E(C)}\omega(e)=\sum_{e\in E(G)}\omega(e)|\{C\in\mathcal{C}_{2\ell}(G):e\in E(C)\}|=\sum_{e\in E(G)}\omega(e)N_{e,2\ell}.
\end{equation*}

By \eqref{e1}, Lemma \ref{lem2-1} and the fact that $\frac{4n(n-1-m)}{n-m}=4n-\frac{4n}{n-m}\le 4n-4$, we have
\begin{equation}\label{2}
   \sum_{C\in\mathcal{C}_{2\ell}(G)}\omega(C)\le \frac{2N_{2\ell}}{n-m}\sum_{e\in E(G)}\omega(e)\le \frac{4n(n-1-m)}{n-m}N_{2\ell}\le(4n-4)N_{2\ell}.
\end{equation}
Since $|\mathcal{C}_{2\ell}(G)|=N_{2\ell}$, \eqref{2} implies that there is a cycle $C'\in\mathcal{C}_{2\ell}(G)$ such that
\begin{equation}\label{e2}
\omega(C')\le 4n-4<4n-2.
\end{equation}

Next, we construct a bipartite graph $B$ with bipartition $(E(C'),[2n])$ such that $V(B)=E(C')\cup [2n]$ and $fi\in E(B)$ if and only if $f\in E(G_i)$, where $f\in E(C')$ and $i\in[2n]$. It immediately follows from $E(C')\subseteq E(G)$ that $d_B(f)\ge 1$ for any $f\in E(C')$.
Clearly, a matching in $B$ which covers all vertices in $E(C')$ corresponds to a rainbow coloring of $C'$.

If there is no such matching in $B$, by Theorem \ref{thm2-1}, there exists a nonempty set $S\subseteq E(C')$ such that $|N_B(S)|\le |S|-1$. It follows that $|S|\le 2\ell\le 2n$, and $|S|\geq2$ since $d_B(f)\ge 1$ for any $f\in E(C')$ implies $|N_B(S)|\ge 1$.
Let $Q=[2n]\setminus N_B(S)$. Then $|Q|\ge 2n-|S|+1\ge1$.

On the other hand, for each $f\in S\subseteq E(C')$ and each $i\in Q$, we have $fi\notin E(B)$, say, $f\in E(\widehat{G_i})$, thus $S\subseteq E(\widehat{G_i})$ for each $i\in Q$, which implies $\omega(C')\ge |S||Q|$, and $|S|\le n-1$ since $e(\widehat{G_i})\le n-1$ for each $i\in Q$.
Therefore,
\[\omega(C')\ge |S||Q|\ge |S|(2n-|S|+1)=4n-2+(|S|-2)(2n-|S|-1)\ge 4n-2,\]
contradicting \eqref{e2}. Thus the bipartite graph $B$ has a matching which covers all vertices in $E(C')$, which implies $C'\in \mathcal{C}_{2\ell}(G)$ is rainbow. It follows that $\mathbf{G}$ contains a rainbow $2\ell$-cycle, and thus (ii) holds.

Combining Case 1 and Case 2, the proof of (ii) is completed.

Clearly, by (ii), (iii) holds.

This completes the proof of Lemma \ref{lem2-2}.
\end{proof}

We now prove Theorem \ref{thm1-7}.

\begin{proof}[\noindent\textbf{Proof of Theorem \ref{thm1-7}}]
Since $G=\bigcup\limits_{i=1}^{2n} G_i$, $e(G)\ge e(G_i)\ge n^2-n+1$. By Theorem \ref{thm1-4}, $G$ is bipancyclic unless $G\cong B_{n}^1$.

If $G$ is bipancyclic, then $\mathbf{G}$ is rainbow bipancyclic by (iii) of Lemma \ref{lem2-2}.

If $G\cong B_{n}^1$, then $e(G)=n^2-n+1$. Since $G_i\subseteq G$ and $e(G_i)\ge n^2-n+1$, we have $G_i=G$ for each $i\in [2n]$, which implies that $G_1=G_2=\dots=G_{2n}\cong B_{n}^1$. Clearly, $G$ is not bipancyclic, and thus $\mathbf{G}$ is not rainbow bipancyclic.

Therefore, the proof of Theorem \ref{thm1-7} is completed.
\end{proof}

\section{Proof of Theorem \ref{thm1-8}}
In this section, we show Theorem \ref{thm1-8}. To this end, we first present some lemmas.

\begin{lemma}\label{lem3-1}
Let $n\ge 2$ and $G$ be a balanced bipartite graph of order $2n$ satisfying $\rho(G)\ge \sqrt{n(n-1)}$. Then either
$G\cong K_{n-1,n}\cup K_1$ or $e(G)\ge n^2-n+1$.
\end{lemma}
\begin{proof}
Since $\rho(G)\ge \sqrt{n(n-1)}$, by Theorem \ref{thm1-5}, $e(G)\ge n^2-n$, and thus $e(G)\ge n^2-n+1$ or $e(G)=n^2-n$. If $e(G)=n^2-n$, then $\rho^2(G)=e(G)$, it follows from Theorem \ref{thm1-5} that $G\cong K_{a,b}\cup (2n-a-b)K_1$ with $ab=n^2-n$ and $a,b\le n$. This implies $\min\{a,b\}=n-1$ and $\max\{a,b\}=n$, and thus $G\cong K_{n-1,n}\cup K_1$.
\end{proof}

Let $n\ge 2$ and $\mathbf{G}=\{G_1, \dots, G_{2n}\}$ be a collection of not necessarily distinct balanced bipartite graphs of order $2n$ with a common bipartition $V=(X,Y)$ such that $\rho(G_i)\ge \sqrt{n(n-1)}$ for each $i\in [2n]$, and let $G$, $\widehat{G}$, $m$, $\mathcal{C}_{2\ell}(G)$ and $N_{2\ell}$ be as defined in Section 2.

Since $\rho(G_i)\geq \sqrt{n(n-1)}$ for each $i\in [2n]$, by Lemma \ref{lem3-1}, we have $G_i\cong K_{n-1,n}\cup K_1$ or $e(G_i)\ge n^2-n+1$ for any $i\in[2n]$.
Now we define the sets $R$ and $T$ as follows:
\[
R=\{i\in [2n]: G_i\cong K_{n-1,n}\cup K_1\},\quad T=[2n]\setminus R.
\]
Let $|R|=r$ and $|T|=t$. Then $r+t=2n$.

Clearly, for each $i\in T$, $e(G_i)\ge n^2-n+1$ and $e(\widehat{G_i})\le n-1$. Recalling the definition of $B_v$, for each $i\in R$, there exists a vertex $v\in V$ such that $G_i=B_v$, for convenience, we write $v=v_i$. In particular, if $G_i=G_j$ for $i,j\in R$, we set $v_i=v_j$.

For each $e\in E(G)$, we define $\omega_T(e)=|\{i\in T:e\notin E(G_i)\}|$. For a subgraph $H$ of $G$, let $ \omega_T(H)=\sum\limits_{e\in E(H)}\omega_T(e)$.
Similar to the proof of Lemma \ref{lem2-2}, we have the following results.

\begin{lemma}\label{lem3-2}
Let $t>0$ and $N_{2\ell}>0$ for $2\le \ell\le n$. Then $m\le n-1$, and there exists a cycle $C\in\mathcal{C}_{2\ell}(G)$ such that $\omega_T(C)<2t$.
\end{lemma}
\begin{proof}
Clearly, $T\neq\emptyset$ as $t>0$.
Recalling that $E(\widehat{G})=E(K_{n,n})\setminus E(G)=\bigcap\limits_{i=1}^{2n} E(\widehat{G_i})$ with $e(\widehat{G})=m$, we have $E(\widehat{G})\subseteq E(\widehat{G_j})$ for any $j\in T$, and thus $m\le e(\widehat{G_j})\le n-1$.

By the definition of $\omega_T(e)$ and the fact that $|T|=t$, we have
\begin{equation}\label{e3}
\sum_{e\in E(G)}\omega_T(e)=\sum_{j\in T}|E(\widehat{G_j})\setminus E(\widehat{G})|\le t(n-1-m).
\end{equation}

If $m=n-1$, then $G_j=G$ for each $j\in T$, and thus $\omega_T(e)=0$ for any $e\in E(G)$. Therefore, $\omega_T(C)=0<2t$ for any cycle $C$ in $G$, and thus the conclusion follows.

If $m\le n-2$, by Lemma \ref{lem2-1} and \eqref{e3}, we have
\begin{align*}
\sum_{C\in\mathcal{C}_{2\ell}(G)}\omega_T(C)
&=\sum_{C\in\mathcal{C}_{2\ell}(G)}\sum_{e\in E(C)}\omega_T(e)
=\sum_{e\in E(G)}\omega_T(e)|\{C\in\mathcal{C}_{2\ell}(G): e\in E(C)\}|\\
&=\sum_{e\in E(G)}\omega_T(e)N_{e,2\ell}
\le\frac{2N_{2\ell}}{n-m}\sum_{e\in E(G)}\omega_T(e)
\le\frac{2t(n-1-m)}{n-m}N_{2\ell}\\
&<2tN_{2\ell}.
\end{align*}
It follows that there exists a cycle $C\in\mathcal{C}_{2\ell}(G)$ such that $\omega_T(C)<2t$.

This completes the proof.
\end{proof}

\begin{lemma}\label{lem3-3}
If $N_{2\ell}>0$ for $2\le\ell\le n$, then one of the following holds:
\begin{enumerate}[label=(\roman*), font=\upshape, itemsep=2pt, align=left, leftmargin=1em]
 \item $\mathbf{G}$ contains a rainbow $2\ell$-cycle;
 \item $\ell=n$ and there exists a vertex $v\in V$ such that $|\{i\in[2n]:G_i=B_v\}|=2n-1$.
\end{enumerate}
\end{lemma}

\begin{proof}
For any cycle $C$ in $G$, similar to the proof of Lemma \ref{lem2-2}, we construct a bipartite graph $B_{C}$ with bipartition $(E(C),[2n])$, where $e\in E(C)$ is adjacent to $i\in[2n]$ if and only if $e\in E(G_i)$. Since $E(C)\subseteq E(G)$, we have $d_{B_{C}}(e)\ge 1$ for any $e\in E(C)$.

\vspace{6pt}
$\mathbf{Case~1.}$ There exists some $C\in\mathcal{C}_{2\ell}(G)$ such that $B_{C}$ contains a matching covering all vertices in $E(C)$.
\vspace{6pt}

Since a matching in $B_{C}$ which covers all vertices in $E(C)$ corresponds to a rainbow coloring of $C$, it follows that $C$ is a rainbow $2\ell$-cycle in $\mathbf{G}$, and thus (i) holds.

\vspace{6pt}
$\mathbf{Case~2.}$ There is no such matching for any $C\in\mathcal{C}_{2\ell}(G)$.
\vspace{6pt}

For any $C\in\mathcal{C}_{2\ell}(G)$, there is a nonempty set $S\subseteq E(C)$ such that $|N_{B_{C}}(S)|\le |S|-1\le 2\ell-1\le 2n-1$ by Theorem \ref{thm2-1}. Since $|N_{B_{C}}(S)|>0$, we have $|S|\ge 2$. Let $Q=[2n]\setminus N_{B_{C}}(S)$. Then $|Q|\ge 2n-|S|+1\ge 1$.

Recall the sets $R$ and $T$ defined above, where $0\le |T|=t\le 2n$.
Now we consider the following three subcases according to the size of $t$.

\vspace{6pt}
$\mathbf{Subcase~2.1.}$ $t=2n$.
\vspace{6pt}

It is clear that $e(G_i)\ge n^2-n+1$ for each $i\in [2n]$. Since $N_{2\ell}>0$, $G$ contains a $2\ell$-cycle, and it follows from (ii) of Lemma \ref{lem2-2} that $\mathbf{G}$ contains a rainbow $2\ell$-cycle. Thus (i) holds.

\vspace{6pt}
$\mathbf{Subcase~2.2.}$ $t=0$.
\vspace{6pt}

It immediately follows that $r=2n$ and $G_i=B_{v_i}$ for some $v_i\in V$ and any $i\in[2n]$. Choose a cycle $C\in\mathcal{C}_{2\ell}(G)$, and let $S$ and $Q$ be defined as above.

We first show $|S|=2$. Assume that $|S|\ge 3$. For any $i\in[2n]$, we know $G_i=B_{v_i}(\cong K_{n-1,n}\cup K_1)$ for some $v_i\in V$. If $v_i\notin V(C)$, then $E(C)\subseteq E(G_i)$, and thus $|E(G_i)\cap S|=|S|\ge 3$ as $S\subseteq E(C)$. If $v_i\in V(C)$, then $|E(G_i)\cap S|\ge |S|-2\ge 1$. It follows that $|N_{B_{C}}(S)|=2n$, a contradiction. Thus $|S|\le 2$, and so $|S|=2$.

Since $|N_{B_{C}}(S)|>0$ and $|N_{B_{C}}(S)|\le |S|-1=1$, we have $|N_{B_{C}}(S)|=1$, and thus $|Q|=2n-1$, which implies there are $2n-1$ graphs in $\mathbf{G}$ missing the two edges of $S$ and $G_k$ contains the two edges of $S$, where $\{k\}=N_{B_{C}}(S)$. Since $G_i=B_{v_i}$ for some $v_i\in V$ and any $i\in[2n]$, we can obtain that there exists a vertex $v\in V$ such that $|\{i\in[2n]:G_i=B_v\}|=2n-1$. Without loss of generality, we take $G_1=G_2=\dots=G_{2n-1}=B_v$.

If $\ell=n$, then (ii) holds. If $2\le\ell\le n-1$, then there exists a rainbow $2\ell$-cycle in $\{G_1, \dots, G_{2n-1}\}$, and thus (i) holds.

\vspace{6pt}
$\mathbf{Subcase~2.3.}$ $1\le t\le 2n-1$.
\vspace{6pt}

Since $r+t=2n$, we have $1\le r\le 2n-1$.
By Lemma \ref{lem3-2}, there exists a cycle $C\in\mathcal{C}_{2\ell}(G)$ such that $\omega_T(C)<2t$. Let $S$ and $Q$ be defined as above. We can obtain the following claims by the property of $C$.

\begin{claimnum}\label{claim1}
$Q\cap R\neq\emptyset$.
\end{claimnum}

\begin{proof}
If not, then $Q\subseteq T$. Thus $|Q|\le t$, and $e(\widehat{G_i})\le n-1$ for each $i\in Q$ by the definition of $T$. Since $N_{B_{C}}(S) \cap Q=\emptyset$, for each $e\in S$ and each $i\in Q$, we have $ei\notin E(B_{C})$, which implies $e\in E(\widehat{G_i})$, and then $e(\widehat{G_i})\ge |S|$ for each $i\in Q$. It follows that $2\le |S|\le n-1$, which implies $n\ge 3$ and $2n-|S|-1\ge 3$.

By the definition of $\omega_T(C)$ and the fact that $Q\subseteq T$, we have
\[
\begin{aligned}
\omega_T(C)&=\sum\limits_{e\in E(C)}|\{i\in T:e\notin E(G_i)\}|\ge |S||Q|\ge |S|(2n-|S|+1)\\&=4n-2+(|S|-2)(2n-|S|-1)\ge 4n-2\ge 2t.
\end{aligned}
\]
This contradicts $\omega_T(C)<2t$.
\end{proof}

By Claim \ref{claim1}, we take $j\in Q\cap R$. Then $G_j=B_{v_j}$ for some $v_j\in V$, and thus $|E(C)\setminus E(G_j)|=\begin{cases}
0, & \text{if } v_j\notin V(C);\\
2, & \text{if } v_j\in V(C).
\end{cases}$ Combining this with $j\in Q$, we have $S\subseteq E(C)\cap E(\widehat{G_j})$, then $|S|=2$ as $|S|\ge2$, and thus $|N_{B_{C}}(S)|=1$.
Without loss of generality, we take $S=\{v_ju,v_jw\}$ and $N_{B_{C}}(S)=\{k\}$, where $j\neq k$. Then $S\subseteq E(G_k)$.

\begin{claimnum}\label{claim2}
For each $i\in R$, $G_i=G_j=B_{v_j}$.
\end{claimnum}

\begin{proof}
Assume for contradiction that there exists some $i\in R$ such that $G_{i}=B_{v_{i}}$, where $v_{i}\neq v_j$.

Now we consider $v_i$ and show that $v_{i}\notin\{u,w\}$. Otherwise, since $G_i-v_i\cong K_{n-1,n}$, we have $v_jw\in E(G_{i})$ if $v_{i}=u$, and $v_ju\in E(G_{i})$ if $v_{i}=w$. It follows that $i\in N_{B_{C}}(S)=\{k\}$, and thus $i=k$, which contradicts $S\subseteq E(G_k)$.

Since $v_{i}\notin\{u,w,v_j\}$ and $G_i-v_i\cong K_{n-1,n}$, we have $\{v_ju,v_jw\}\subseteq E(G_{i})$, and thus $i\in N_{B_{C}}(S)=\{k\}$, which implies $i=k$. It follows that $T\subseteq Q$. Therefore, $\omega_T(C)=\sum\limits_{e\in E(C)}|\{i\in T:e\notin E(G_i)\}|\ge |S||T|=2t$. This contradicts $\omega_T(C)<2t$.
\end{proof}

Since $S\cap E(G_j)=\emptyset$ and $S\subseteq E(G_k)$, by Claim \ref{claim2}, we have $k\in T$. Thus $\omega_T(v_ju)=\omega_T(v_jw)=t-1$. Let $P=C-v_j=w_1w_2\cdots w_{2\ell-1}$, where $w_1=u$ and $w_{2\ell-1}=w$. Then
\begin{equation}\label{e4}
\omega_T(P)=\omega_T(C)-\omega_T(v_ju)-\omega_T(v_jw)<2t-2(t-1)=2.
\end{equation}

$\mathbf{Subcase~2.3.1.}$ $t=1$.

In this subcase, $r=2n-1$. By Claim \ref{claim2}, $|\{i\in[2n]:G_i=B_{v_j}\}|=2n-1$. If $\ell=n$, then (ii) holds. If $2\le\ell\le n-1$, then there exists a rainbow $2\ell$-cycle in $\mathbf{G}\setminus\{G_k\}$, and thus (i) holds.

$\mathbf{Subcase~2.3.2.}$ $2\le t\le 2n-1$.

Let $p\in T\setminus\{k\}$. Then $p\in Q\cap T$ since $N_{B_{C}}(S)=\{k\}$, which implies $S\cap E(G_p)=\emptyset$ and $e(G_p)\ge n^2-n+1$, and thus $u,w\notin N_{G_p}(v_j)$ and $d_{G_p}(v_j)\ge 1$. Thus $n\ge 3$, and we take $z\in N_{G_p}(v_j)$. Then $u,w,z$ belong to the same partite set.

Now we show there is a path $P'$ of order $2\ell-1$ in $G$ from $w_1(=u)$ to $z$ such that $P'$ differs from $P$ in exactly one edge.
If $z=w_a\in V(P)$, then $a$ is odd and $3\le a\le 2\ell-3$, and thus  $P'=w_1w_2\cdots w_{a-1}w_{2\ell-1}w_{2\ell-2}\cdots w_a$ is the desired path since $w_{a-1}w_{2\ell-1}\in E(G_i)$ for each $i\in R$ by Claim \ref{claim2}.
If $z\notin V(P)$, then $P'=w_1w_2\cdots w_{2\ell-2}z$ is the desired path since $w_{2\ell-2}z\in E(G_i)$ for each $i\in R$ by Claim \ref{claim2}.

Let $C'=v_jw_1P'zv_j$. Then $C'\in\mathcal{C}_{2\ell}(G)$. In the following, we show that $C'$ is a rainbow $2\ell$-cycle. Let $e'$ be the unique edge in $E(P')\setminus E(P)$ and $q\in R$. Then $e'\in E(G_q)$ by Claim \ref{claim2}. Note that $v_jw_1\in E(G_k)$ and $v_jz\in E(G_p)$, where $k,p\in T$.
By the definition of $P',C'$, we have $E(C')\setminus\{v_jw_1,v_jz,e'\}=E(P)\cap E(P')$.

Now we construct a bipartite graph $B'$ with bipartition $(E(P)\cap E(P'),[2n]\setminus\{k,p,q\})$, where $e\in E(P)\cap E(P')$ is adjacent to $i\in[2n]\setminus\{k,p,q\}$ if and only if $e\in E(G_i)$. By Claim \ref{claim2}, we have $ei\in E(B')$ for any $e\in E(P)\cap E(P')$ and any $i\in R\cap ([2n]\setminus\{k,p,q\})$. Combining this with \eqref{e4}, it follows that $B'$ is a graph obtained from $K_{2\ell-3,2n-3}$ by deleting at most one edge, where $2n-3\ge 3$ as $n\ge 3$. By Observation \ref{ob2-1}, $B'$ has a matching covering all vertices in $E(P)\cap E(P')$. Therefore, $C'$ is a rainbow $2\ell$-cycle, and thus (i) holds.

Combining the above arguments, either (i) or (ii) holds, and thus the proof is completed.
\end{proof}

Now we give the proof of Theorem \ref{thm1-8}.
\begin{proof}[\noindent\textbf{Proof of Theorem \ref{thm1-8}}]
We first prove the sufficiency.
If (i) or (ii) holds, then $v$ is not contained in any rainbow cycle, and thus $\mathbf{G}$ contains no rainbow $2n$-cycle, which implies that $\mathbf{G}$ is not rainbow bipancyclic.

Next, we prove the necessity. Assume that $\mathbf{G}$ is not rainbow bipancyclic. Since $\rho(G_i)\ge \sqrt{n(n-1)}$ for each $i\in [2n]$, by Theorem \ref{thm1-5}, we have $e(G_i)\ge n^2-n$ for each $i\in [2n]$, and thus $e(G)\ge n^2-n$. We distinguish two cases according to $e(G)$.

\vspace{6pt}
$\mathbf{Case~1.}$ $e(G)=n^2-n$.
\vspace{6pt}

Since $G_i\subseteq G$ and $e(G_i)\ge n^2-n$ for each $i\in [2n]$, we have $G_1=G_2=\dots=G_{2n}=G$. By Lemma \ref{lem3-1}, $G=B_v\cong K_{n-1,n}\cup K_1$ for some $v\in V$, which implies $|\{i\in[2n]:G_i=B_v\}|=2n$. Thus both (i) and (ii) hold.

\vspace{6pt}
$\mathbf{Case~2.}$ $e(G)\ge n^2-n+1$.
\vspace{6pt}

By Theorem \ref{thm1-4}, $G$ is bipancyclic unless $G\cong B_{n}^1$.

If $G$ is bipancyclic, then $N_{2\ell}>0$ for each $2\le\ell\le n$, and
by Lemma \ref{lem3-3}, either $\mathbf{G}$ contains a rainbow $2\ell$-cycle for each $2\le\ell\le n$, or there exists a vertex $v\in V$ such that $|\{i\in[2n]:G_i=B_v\}|=2n-1$. Since $\mathbf{G}$ is not rainbow bipancyclic, (i) holds.

If $G\cong B_{n}^1$, then there exist distinct $u,v\in V$ such that $G_i\subseteq G=B_{v,u}$ for each $i\in [2n]$. For $n=2$, we have $G_i\in\{B_v,B_w,B_{v,u}\}$ for each $i\in [4]$, where $w\,(\neq u)$ is in the same partite set as $u$.
For $n\ge 3$, if $G_i\neq B_{v,u}$ for some $i\in [2n]$, then $e(G_i)<e(B_{v,u})=n^2-n+1$. Since $\rho(G_i)\ge \sqrt{n(n-1)}$, by Lemma \ref{lem3-1}, we have $G_i=B_{v_i}$ for some $v_i\in V$. It follows that $G_i=B_{v}$ since $G_i\subseteq B_{v,u}$ and $n\ge 3$. Therefore, $G_i\in\{B_v,B_{v,u}\}$ for $n\ge 3$ and each $i\in [2n]$. Thus (ii) holds.

Therefore, the proof of Theorem \ref{thm1-8} is completed.
\end{proof}

\section*{Funding}
This work is supported by the National Natural Science Foundation of China (Grant Nos. 12371347, 12271337).

\section*{Declarations}
\noindent\textbf{Conflict of interest} The authors declare that they have no known competing financial interests or personal relationships that could have appeared to influence the work reported in this paper.\\
\textbf{Data availability} No data was used for the research described in the article.

\end{spacing}
\end{document}